\documentclass{amsart}
\usepackage{mathrsfs}

\theoremstyle{plain}
\newtheorem{thm}{Theorem}[section]
\newtheorem{lemma}[thm]{Lemma}

\newtheorem{prop}[thm]{Proposition}

\theoremstyle{definition}
\newtheorem{definition}[thm]{Definition}
\newtheorem{remark}[thm]{Remark}

\newtheorem{example}[thm]{Example}

\catcode`\@=11
\def\mequal{\mathrel{\mathpalette\@mvereq{\hbox{\sevenrm m}}}} 
\def\@mvereq#1#2{\lower.5\p@\vbox{\baselineskip\z@skip\lineskip1.5\p@
    \ialign{$\m@th#1\hfil##\hfil$\crcr#2\crcr=\crcr}}}
\def\partr#1#2{/\kern-.08333em/_{#1,#2}^{\phantom{.}}}
\def\invpartr#1#2{/\kern-.08333em/_{#1,#2}^{-1}} 
\def\hpartr#1#2{/\kern-.08333em/_{#1,#2}^{h}}
\def\Epartr#1#2{/\kern-.08333em/_{#1,#2}^{E}}

\def\newdot{{\kern.8pt\cdot\kern.8pt}}

\def\,{\relax\ifmmode\mskip\thinmuskip\else\thinspace\fi}
\def\{{\relax\ifmmode\lbrace\else $\lbrace$\fi}
\def\}{\relax\ifmmode\rbrace\else $\rbrace$\fi}

\font\sevenrm=cmr7

\newcommand\CC{\mathbb{C}}

\newcommand\EE{\mathbb{E}}

\newcommand\RR{\mathbb{R}}

\newcommand\TT{\mathbb{T}}
\newcommand\ZZ{\mathbb{Z}}

\newcommand{\SD}{{\mathscr D}}

\newcommand{\SG}{{\mathscr G}}

\newcommand{\SH}{{\mathscr H}}

\newcommand{\SP}{{\mathscr P}}

\def\mathpal#1{\mathop{\mathchoice{\text{\rm #1}}%
   {\text{\rm #1}}{\text{\rm #1}}%
   {\text{\rm #1}}}\nolimits}

\def\Ric{\mathpal{Ric}}

\def\id{\mathpal{id}}
\def\trace{\mathpal{tr}}

\def\di{\displaystyle}
\def\f{\frac}
\def\a{\alpha }

\def\D{\Delta }
\def\d{\delta }
\def\e{\varepsilon }
\def\G{\Gamma }
\def\g{\gamma }
\def\l{\lambda }
\def\n{\nabla }
\def\Om{\Omega }
\def\om{\omega }

\def\s{\sigma }

\begin{document}

\title[Navier-Stokes equation and diffusions of homeomorphisms]{Lagrangian Navier-Stokes  diffusions on manifolds: variational principle and stability}

\author[M. Arnaudon]{Marc Arnaudon} \address{Laboratoire de Math\'ematiques et
  Applications\hfill\break\indent CNRS: UMR 6086\hfill\break\indent
  Universit\'e de Poitiers, T\'el\'eport 2 - BP 30179\hfill\break\indent
  F--86962 Futuroscope Chasseneuil Cedex, France}
\email{marc.arnaudon@math.univ-poitiers.fr}

\author[A. B. Cruzeiro]{Ana Bela Cruzeiro} \address{GFMUL and Dep. de Matem\'atica IST(TUL). \hfill\break\indent
  Av. Rovisco Pais\hfill\break\indent
  1049-001 Lisboa, Portugal}
\email{abcruz@math.ist.utl.pt}

%
%

\begin{abstract}\noindent

We prove a variational principle for stochastic Lagrangian Navier-Stokes trajectories
on manifolds. We study the behaviour of such trajectories concerning stability as well as
rotation between particles; the two-dimensional torus case is described in detail.
  
\end{abstract}

\maketitle
\tableofcontents

%
%

\section{Introduction}\label{Section1}
\setcounter{equation}0

As discovered by V. I. Arnold (\cite{Arnold:66}) the motion of an incompressible non viscous fluid can be characterized as a geodesic
on a group of diffeomorphisms. This  
point of view allows in particular  to derive properties of the Lagrangian Euler flow, such as stability, through 
the study of the geometry of the group (\cite{Misiolek:93}).

When the fluid is  viscous, namely for the Navier-Stokes equation, one can describe the Lagrangian
 trajectories as realizations of a stochastic process and interpret the associated drift, solving Navier-Stokes, as an expectation over this process. This
 intrinsically probabilistic  approach we follow here  is inspired by \cite{Nakagomi:81}, \cite{Yasue:83}. Similar stochastic models are used for
   example in \cite{Constantin-Iyer:08}. Then the trajectories remain, in an appropriate sense, geodesics and are almost sure solutions of
   a variational principle. This was shown in \cite{Cipriano-Cruzeiro:07} for the two-dimensional torus. 

We prove a variational principle for the Lagrangian Navier-Stokes diffusions in a compact Riemannian
 manifold. Furthermore we study its stability properties. The behaviour of the
 trajectories depends on the intensity of the noise as well as on the metric of the underlying
 manifold. The example of the torus is studied in detail. Finally we describe the evolution in time of the rotation between stochastic Lagrangian particles.

Let $(M,\bf g \rm )$ be a compact oriented Riemannian manifold without boundary.

Recall that the It\^o differential of an $M$-valued semimartingale $Y$ is  defined by 
\begin{equation}
 \label{E1}
dY_t=P\left(Y\right)_td\left(\int_0^\cdot P\left(Y\right)_s^{-1}\circ dY_s\right)_t
\end{equation}
where 
\begin{equation}\label{E2}
P\left(Y\right)_t : T_{Y_0}M\to T_{Y_t}M
\end{equation}
is the parallel transport along $t\mapsto Y_t$. Alternatively, in local coordinates, 
\begin{equation}\label{E3}
dY_t=\left(dY_t^i+\f12\G_{jk}^i(Y_t)dY_t^j\otimes dY_t^k\right)\partial_i
\end{equation}
where $\G_{jk}^i$ are the Christoffel symbols of the Levi-Civita connection. 

If the semimartingale $Y_t$ has an absolutely continuous drift, we denote it by $DY_t\,dt$: for every $1$-form $\a\in\G(T^\ast M)$, the finite variation part of 
\begin{equation}\label{E4}
\int_0^\cdot \left\langle \a(Y_t),\ dY_t\right\rangle
\end{equation}
is 
\begin{equation}\label{E5}
\int_0^\cdot \left\langle \a(Y_t),\ DY_t\,dt\right\rangle
\end{equation}

Let $G^s$, $s\ge 0$ be the infinite dimensional group of homeomorphisms on $M$ which belong to $H^s$, the Sobolev space of order~$s$. For $s>\f{m}2+1$, $m={\rm dim}M$, $G^s$ is a $C^\infty$ Hilbert manifold. The volume preserving homeomorphism subgroup will be denoted by $G_V^s$: 
$$
G_V^s=\{g\in G^s, \ :\ g_\ast \mu=\mu\},
$$
with $\mu$ the volume element associated to the Riemannian metric. We denote by $\SG^s$ (resp. $\SG_V^s$) the Lie algebra of $G^s$ (resp. $G_V^s$). See \cite{Misiolek:93} for example.

On $M$ we consider an incompressible Brownian flow $g_u(t)\in G^0_V$ with covariance $a\in \G(TM\odot TM)$ and time dependent  drift $u(t,\cdot)\in \G(TM)$. 
We assume that for all $x\in M$, $a(x,x)=2\nu {\bf g }^{-1}(x)$ for some $\nu>0$. This means that  
\begin{equation}\label{E6}
dg_u(t)(x)\otimes dg_u(t)(y)=a\left(g_u(t)(x), g_u(t)(y)\right)\,dt,
\end{equation}
\begin{equation}
 \label{E7}
dg_u(t)(x)\otimes dg_u(t)(x)=2\nu {\bf g }^{-1}\left(g_u(t)(x)\right)\,dt,
\end{equation}
the drift of $g_u(t)(x)$ is absolutely continuous and satisfies $Dg_u(t)(x)=u(t,g_u(t)(x))$. The generator of this process is

$$L_u =\nu \Delta^h +\f{\partial}{\partial t} +\partial_u $$
where $\D^h $ is the horizontal Laplacian.
The parameter $\nu$ will be called the speed of the Brownian flow.

If the time is indexed by $[0,T]$ for some $T>0$, we define the action functional by 
$$
S(g_u)=\f12 \EE\left[\int_0^T\left(\int_M\left\|Dg_u(t)(x)\right\|^2\,dx\right)\,dt\right].
$$

\section{The variational principle}\label{Section2}
\setcounter{equation}0

Define 
\begin{equation}
 \label{E8}
\SH=\left\{v\in C^1([0,T],\ \SG_V^\infty),\ v(0,\cdot)= 0,\ v(T,\cdot)=0\right\}
\end{equation}

Given $v\in \SH$, consider the following ordinary differential equation
\begin{align}
 \label{E9}
\begin{split}
 \f{de_t(v)}{dt}&=\dot v(t,e_t(v))\\
e_0(v)&=e
\end{split}
\end{align}
where $e$ is the identity of $G_V^\infty$.
Since $v$ is divergence free, $e_\cdot (v)$ is a $G_V^\infty$-valued deterministic path. 

We denote by $\SP$ the set of continuous $G_V^0$-valued semimartingales $g(t)$ such that $g(0)=e$. 
Then for all $v\in \SH$, we have  $e_t(v)\circ g_u(t)\in \SP$.

\begin{definition}
 \label{D1}
Let $J$ be a functional defined on $\SP$ and taking values in $\RR$. We define its left and right derivatives in the direction of $h(\cdot)=e_\cdot (v)$, $v\in \SH$ at a process $g\in \SP$ respectively, by 
\begin{align}
 \label{E10}
\begin{split}
 (D_L)_hJ[g]&=\f{d}{d\e}J[e_\cdot (\e v)\circ g(\cdot)]\vert_{\e=0},\\
(D_R)_hJ[g]&=\f{d}{d\e}J[g(\cdot)\circ e_\cdot (\e v)]\vert_{\e=0}.
\end{split}
\end{align}
A process $g\in \SP$ wil be called a critical point of the functional $J$ if 
\begin{equation}
 \label{E11}
(D_L)_hJ[g]=(D_R)_hJ[g], \ \forall h=e(v),\ v\in \SH.
\end{equation}
\end{definition}

\begin{thm}
 \label{T1}
Let $(t,x)\mapsto u(t,x)$ be a smooth time-dependent divergence-free vector field on $M$, defined on $[0,T]\times M$.  Let $g_u(t)$ a stochastic Brownian flow with speed $\nu>0$ and drift $u$. The stochastic process $g_u(t)$ is a critical point of the energy functional $S$ if and only if the vector field $u(t)$ verifies the Navier-Stokes equation
\begin{equation}
 \label{E12}
\f{\partial u}{\partial t}+\n_uu=\nu \square u-\n p.
\end{equation}
\end{thm}

For the construction of weak  solutions of Navier-Stokes equations on Riemannian manifolds we refer to
\cite{Nagasawa:97}.
\begin{proof}

 Since the functional $S$ is right invariant, it is enough to consider the left derivative. 
So we need to compute 
\begin{equation}
 \label{E22}
\f{d}{d\e}\vert_{\e=0} S(e_\cdot(\e v)(g_u)).
\end{equation}
We let 
\begin{equation}
 \label{E23}
f(\e)=S(e_\cdot(\e v)(g_u)).
\end{equation}
Then 
\begin{equation}
 \label{E24}
f(\e)=\f12\int_M\left(\EE\left[\int_0^T\left(\left\|De_t(\e v)(g_u)(t)(x)\right\|^2\right)\,dt\right]\right)\,dx
\end{equation}
which yields 
\begin{equation}
 \label{E26}
f'(0)=\int_M\left(\EE\left[\int_0^T\left(\left\langle \n_\e\vert_{\e=0}De_t(\e v)\left( g_u(t)(x)\right), u(t,g_u(t)(x))\right\rangle\right)\,dt\right]\right)\,dx.
\end{equation}

We need to compute 
\begin{equation}
 \label{E25}
\n_\e\vert_{\e=0}De_t(\e v)\left( g_u(t)(x)\right).
\end{equation}
We have 
\begin{align*}
 \n_t\f{d}{d\e}\vert_{\e=0}e_t(\e v)&=\n_\e\vert_{\e=0}\f{de_t(\e v)}{dt}&\\
&=\n_\e\vert_{\e=0} \e\dot v(t,e_t(\e v))\\
&=\dot v(t,e).
\end{align*}
Together with $v(0,\cdot)=0$, this implies
\begin{align}
 \label{E13}
\f{d}{d\e}\vert_{\e=0}e_t(\e v)(x)=v(t,x).
\end{align}
Consequently 
\begin{equation}
 \label{E14}
\f{d}{d\e}\vert_{\e=0}e_t(\e v)\left( g_u(t)(x)\right)=v\left(t, g_u(t)(x)\right).
\end{equation}
By It\^o equation,
 \begin{align}
  \label{E15}
\begin{split}
&de_t(\e v)(g_u(t)(x))\\&=\langle de_t(\e v)(\cdot),\ dg_u(t)(x)\rangle+\f12 \n de_t(\e v)(g_u(t)(x))\left(dg_u(t)(x)\otimes dg_u(t)(x)\right)\\
&=\langle de_t(\e v)(\cdot),\ dg_u(t)(x)\rangle+ \nu \D e_t(\e v)(g_u(t)(x))\,dt.
\end{split}
 \end{align}
Here $\D e_t(\e v)(\cdot)$ denotes the tension field of the map $e_t(\e v) : M\to M$.
This yields
\begin{align}
  \label{E16}
\begin{split}
De_t(\e v)(g_u(t)(x))&=\langle de_t(\e v)(\cdot),\ Dg_u(t)(x)\rangle+ \nu \D e_t(\e v)(g_u(t)(x))\\&\qquad+\e \dot v(t,e_t(\e v)(g_u(t)(x))) \\
&=\langle de_t(\e v)(\cdot),\ u(t,g_u(t)(x))\rangle+ \nu \D e_t(\e v)(g_u(t)(x))\\&\qquad+\e \dot v(t,e_t(\e v)(g_u(t)(x))).
\end{split}
 \end{align}
Differentiating with respect to $\e$ at $\e=0$, we get

\begin{align}
  \label{E17}
\begin{split}
&\n_\e\vert_{\e=0}De_t(\e v)(g_u(t)(x))\\&=\left\langle\n_\e\vert_{\e=0}de_t(\e v)(\cdot),\ u(t,g_u(t)(x))\right\rangle
+\nu \n_\e\vert_{\e=0}\D e_t(\e v)(g_u(t)(x))\\
&\qquad +\f{\partial v}{\partial t}(t,g_u(t)(x))\\
&=\left\langle\n_\cdot\f{d}{d\e}\vert_{\e=0}e_t(\e v)(\cdot),\ u(t,g_u(t)(x))\right\rangle
+\nu \square \f{d}{d\e}\vert_{\e=0}e_t(\e v)(g_u(t)(x))\\
&\qquad +\f{\partial v}{\partial t}(t,g_u(t)(x))\\
&=\left\langle\n_\cdot v(t,\cdot),\ u(t,g_u(t)(x))\right\rangle
+\nu \square v(t,\cdot) (g_u(t)(x)) +\f{\partial v}{\partial t}(t,g_u(t)(x))\\
&=\n_{u(t,g_u(t)(x))} v(t,\cdot)+\nu \square v(t,\cdot) (g_u(t)(x)) +\f{\partial v}{\partial t}(t,g_u(t)(x)).
\end{split}
 \end{align}
We used the commutation formula $\n_\e\vert_{\e=0}\D=\square \f{d}{d\e}$, where $\square=dd^\ast +d^\ast d$ is the damped Laplacian. Alternatively, \begin{equation}\label{E33}\square v=\D^h v+\Ric^\sharp(v).\end{equation}

For a $TM$-valued semimartingale $J_t$ which projects onto the $M$-valued semimartingale $Y_t$, we denote by $\SD J_t$ the It\^o covariant derivative:
\begin{equation}
 \label{E18}
\SD J_t=P(Y)_td\left( P(Y)_t^{-1} J_t\right).
\end{equation}
Then It\^o equation yields 
\begin{equation}
 \label{E19}
\SD u(t,g_u(t)(x))\simeq \f{\partial u}{\partial t}(t, g_u(t)(x))\,dt+ \n_{dg_u(t)(x)}u+\nu \D^h u(t,g_u(t)(x))\,dt
\end{equation}
and 
\begin{equation}
 \label{E20}
\SD v(t,g_u(t)(x))\simeq \f{\partial v}{\partial t}(t, g_u(t)(x))\,dt+ \n_{dg_u(t)(x)}v+\nu \D^h v(t,g_u(t)(x))\,dt.
\end{equation}
where the notation $\simeq$ means equal up to a martingale: 
\begin{align*}
&\int_0^\cdot P(g_u(\cdot))_t^{-1}\SD u(t,g_u(t)(x))\\&-\int_0^\cdot P(g_u(\cdot))_t^{-1}\left(\f{\partial u}{\partial t}(t, g_u(t)(x))\,dt+ \n_{dg_u(t)(x)}u+\nu \D^h u(t,g_u(t)(x))\,dt\right)
\end{align*}
 is a local
martingale.

On the other hand, denoting  $u_t=u(t,g_u(t)(x))$ and $v_t=v(t,g_u(t)(x))$ we have 
\begin{equation}
 \label{E21}
\langle u_T, v_T\rangle=\int_0^T\langle \SD u_t,\ v_t\rangle +\int_0^T\langle u_t,\ \SD v_t\rangle+\int_0^T\langle \SD u_t,\ \SD v_t\rangle.
\end{equation}
Let us denote by $Dv_t$  the drift of $v_t$ with respect to the damped connection $\n^c$ on $TM$, whose geodesics are the Jacobi fields. 
It is known that, 
\begin{equation}
 \label{E30}
\left(D u_t-\nu \Ric^\sharp(u_t)\right)\,dt \quad\hbox{ is the drift of}\quad \SD u_t
\end{equation}
 and 
\begin{equation}
 \label{E31}
\left(D v_t-\nu \Ric^\sharp(v_t)\right)\,dt \quad\hbox{ is the drift of}\quad  \SD v_t.
\end{equation}
As can be seen from \eqref{E17}, \eqref{E20} and \eqref{E31}, the drift 
$Dv_t$ commutes with the derivative with respect to a parameter, so it satisfies 
\begin{equation}
 \label{E29}
D v_t=\n_\e\vert_{\e=0}D e_t(\e v)(g_u(t)(x)). 
\end{equation}

Taking the expectation in~\eqref{E21} and using~\eqref{E29},~\eqref{E30} and~\eqref{E31}, we get by removing the martingale parts 
\begin{align}
 \label{E27}
\begin{split}
\EE\left[\langle u_T, v_T\rangle\right]&=\EE\left[\int_0^T\langle \f{\partial u}{\partial t}(t, g_u(t)(x))+ \n_{u_t}u+\nu \D^h u(t,g_u(t)(x)),\ v_t\rangle\,dt\right]\\&+ \EE\left[\int_0^T\langle u_t,\  \n_\e\vert_{\e=0}D e_t(\e v)(g_u(t)(x))-\nu \Ric^\sharp(v_t)\rangle\,dt\right]\\
&+\EE\left[2\nu \int_0^T \trace\left\langle\n_\cdot u,\ \n_\cdot v \right\rangle(t,g_u(t)(x))\,dt\right].
\end{split}
\end{align}
Then using the facts that $v_T=0$, together with
\begin{equation}
 \label{E32}
\langle u_t, \Ric^\sharp(v_t)\rangle=\langle \Ric^\sharp(u_t), v_t\rangle
\end{equation}
and~\eqref{E33}, we get 
\begin{align}
 \label{E34}
\begin{split}
 &\EE\left[\int_0^T\langle u_t,\  \n_\e\vert_{\e=0}D e_t(\e v)(g_u(t)(x))\rangle\,dt\right]\\
&=-\EE\left[\int_0^T\langle \f{\partial u}{\partial t}(t, g_u(t)(x))+ \n_{u_t}u+\nu \square u(t,g_u(t)(x)),\ v_t\rangle\,dt\right]\\&
-\EE\left[2\nu \int_0^T \trace\left\langle\n_\cdot u_t,\ \n_\cdot v_t \right\rangle(t,g_u(t)(x))\,dt\right].
\end{split}
\end{align}
Integrating with respect to $x$ yields
\begin{align}
\label{E35}
\begin{split}
 &f'(0)\\
&=-\EE\left[\int_0^T\left(\int_M\left\langle \left(\left(\f{\partial }{\partial t}+ \n_{u}+\nu \square\right)u\right)(t,g_u(t)(x)),\ v(t,g_u(t)(x))\right\rangle\,dx\right)\,dt\right]\\&
-\EE\left[2\nu \int_0^T \left(\int_M\trace\left\langle\n_\cdot u,\ \n_\cdot v \right\rangle(t,g_u(t)(x))\,dx\right)\,dt\right].
\end{split}
\end{align}
Now we use the fact that  $g_u(t)(\cdot)$ is volume preserving:
\begin{align}
\label{E36}
\begin{split}
 &f'(0)\\
&=-\EE\left[\int_0^T\left(\int_M\left\langle \left(\left(\f{\partial }{\partial t}+ \n_{u}+\nu \square\right)u\right)(t,x),\ v(t,x)\right\rangle\,dx\right)\,dt\right]\\&
-\EE\left[2\nu \int_0^T \left(\int_M\trace\left\langle\n_\cdot u,\ \n_\cdot v \right\rangle(t,x)\,dx\right)\,dt\right].
\end{split}
\end{align}
Since $M$ is compact and orientable,  an integration by parts gives
\begin{equation}
 \label{E37}
\int_M\trace\left\langle\n_\cdot u,\ \n_\cdot v \right\rangle(t,x)\,dx=-\int_M\left\langle\square u, v \right\rangle(t,x)\,dx.
\end{equation}
Replacing in~\eqref{E36} we get 
\begin{align}
\label{E38}
\begin{split}
 &f'(0)=-\EE\left[\int_0^T\left(\int_M\left\langle \left(\left(\f{\partial }{\partial t}+ \n_{u}-\nu \square\right)u\right)(t,x),\ v(t,x)\right\rangle\,dx\right)\,dt\right].
\end{split}
\end{align}
 The process $g_u(t)$ is a critical point of the energy functional $S$ if and only if $f'(0)=0$, which by equation~\eqref{E38} is equivalent to 
\begin{equation}
 \label{E39}
\left(\f{\partial }{\partial t}+ \n_{u}-\nu \square\right)u=-\n p
\end{equation}
for some function $p$ on $[0,T]\times M$. This achieves the proof.
\end{proof}

\section{A martingale characterization for solutions of Navier-Stokes equations}\label{Section3}
\setcounter{equation}0

In this section, to simplify the equations,  we assume the pressure to be constant.
 The pressure will
not be present, in any case, in the weak version of the formulae we derive.

We seek to obtain a formula for the drift of the covariant derivative with respect to a parameter of a family of Navier-Stokes solutions, extending the well-known Jacobi equation.

 Consider a family of diffusions $g^\a$, $\a\in \RR$, satisfying 
\begin{equation}
 \label{E40}
g^\a(0)=\varphi(\a)
\end{equation}
where $\varphi : \RR\to M$ is a smooth path on $M$,
and solution to the It\^o SDE
\begin{equation}
 \label{E41}
dg^\a(t)=u(t,g^\a(t))\,dt+\s (g^\a(t))\,dB_t
\end{equation}
where $u$ solves
\begin{equation}
 \label{E42}
\partial_t u+\n_uu+\nu\square u=0,
\end{equation}
$B_t=(B_t^\ell)_{\ell\ge 0}$ is a family of real Brownian motions, $\s=(\s_\ell)_{\ell\ge 0}$,  and for all $\ell\ge 0$, $\s_\ell$ is a vector field on $M$. We furthermore assume that 
\begin{equation}
 \label{E43}
\s\s^\ast=\nu g^{-1}
\end{equation}
where $g$ is the Riemannian metric on $M$.

We denote by $u_t^\a=Dg^\a(t)=u(t,g^\a(t))$ the drift of $g^\a$. We denote by $\SD^c J_t$ the vertical part of the It\^o differential  (with respect to $\n^c$) of a $TM$-valued semimartingale $J_t$. It is known that 
\begin{align}
 \label{E45}
\SD^c J_t= \SD J_t+\f12R(J_t, dX_t)dX_t
\end{align}
where $X_t=\pi(J_t)$ and $R$ is the curvature tensor. If $J_t$ has an absolutely continuous drift $D^cJ_t$, then the finite variation part of $\SD^c J_t$ is $D^cJ_t\,dt$.

  From the It\^o equation 
\begin{equation}
 \label{E44}
\SD^c u_t^\a \simeq \partial_t u(t, g^\a(t))\,dt+\n_{dg^\a(t)}u+\nu \square u(t, g^\a(t))\,dt
\end{equation}
we deduce that the drift of $\SD^c u_t^\a$ is 
\begin{equation}
 \label{E47}
D^c u_t^\a=\partial_t u(t, g^\a(t))+\n_{u_t^\a}u+\nu \square u(t, g^\a(t))=0.
\end{equation}

>From \cite{Arnaudon-Thalmaier:03} Theorem~4.5, we have formally
\begin{align}
 \label{E60}
\begin{split}
&\SD \n_\a u_t^\a\\=&\n_\a\SD u_t^\a+R(dg^\a(t), \partial_\a g^\a(t))u_t^\a\\&+ R(dg^\a(t), \partial_\a g^\a(t))\SD u_t^\a
-\nu d^\ast R(\partial_\a g^\a(t))u_t^\a +\f12 R(dg^\a(t), \SD\partial_\a g^\a(t))u_t^\a.
\end{split}
\end{align}
Using \eqref{E45}, we obtain
\begin{align}
 \label{E61}
\begin{split}
&\SD^c \n_\a u_t^\a\\=&\n_\a\SD^c u_t^\a+R(dg^\a(t), \partial_\a g^\a(t))u_t^\a-\nu\n_{\partial_\a g^\a(t)}\Ric^\sharp (u_t^\a)\,dt
\\&+ R(dg^\a(t), \partial_\a g^\a(t))\SD u_t^\a
-\nu d^\ast R(\partial_\a g^\a(t))u_t^\a\,dt +\f12 R(dg^\a(t), \SD\partial_\a g^\a(t))u_t^\a.
\end{split}
\end{align}
Removing the martingale part we obtain the drift 
\begin{align}
 \label{E62}
\begin{split}
D^c \n_\a u_t^\a=&\n_\a D^c u_t^\a+R(u_t^\a, \partial_\a g^\a(t))u_t^\a-\nu\n_{\partial_\a g^\a(t)}\Ric^\sharp (u_t^\a)
\\&+ 2\nu \trace R(\cdot , \partial_\a g^\a(t))\n_\cdot  u_t^\a
-\nu d^\ast R(\partial_\a g^\a(t))u_t^\a +\nu \trace R(\s(\cdot) , \n_{\partial_\a  g^\a(t)}\s(\cdot))u_t^\a.
\end{split}
\end{align}
Now since $D u_t^\a=0$ we finally get 
\begin{prop}
\label{P1}
The drift of the covariant derivative with respect to $\a$ of the family $(u_t^\a)_{\a\in\RR}$ of Navier-Stokes solutions is given by
\begin{align}
 \label{E63}
\begin{split}
D^c \n_\a u_t^\a=&R(u_t^\a, \partial_\a g^\a(t))u_t^\a-\nu\n_{\partial_\a g^\a(t)}\Ric^\sharp (u_t^\a)
\\&+ 2\nu \trace R(\cdot , \partial_\a g^\a(t))\n_\cdot  u_t^\a
-\nu d^\ast R(\partial_\a g^\a(t))u_t^\a +\nu \trace R(\s(\cdot) , \n_{\partial_a  g^\a(t)}\s(\cdot))u_t^\a.
\end{split}
\end{align}
\end{prop}
 This formula extends the well known corresponding (Jacobi) equation for the variation of geodesics.

\section{The two-dimensional torus endowed with the Euclidean distance}\label{Section6}
\setcounter{equation}0
 We study the evolution in time of the $L^2$ distance between two particles in the two dimensional torus. Notice that, in order to interpret the diffusion processes as a solution of the variational principle described in section 2, there is no canonical choice for the Brownian motion, as far as
it corresponds to the same generator. We  make here a particular choice.

On the two-dimensional torus $\TT=\RR/2\pi\ZZ\times \RR/2\pi\ZZ$ we consider the following vector fields

$$A_k (\theta )=(k_2 ,-k_1 )\cos k.\theta ,\quad B_k (\theta )=(k_2 ,-k_1 )\sin k.\theta$$
 and the Brownian motion

\begin{equation}
 \label{E67}
dW(t)=\sum_{k \in \ZZ }{\l_k\sqrt \nu}(A_k dx_k +B_k dy_k )
\end{equation}
 where $x_k ,y_k $ are independent copies of real Brownian motions.
We assume that $\sum_k |k|^2 \l_k^2 <\infty$, a necessary and sufficient condition for the Brownian flow to be defined in $L^2 (\TT )$. Furthermore we consider
$\lambda_k =\lambda (|k|)$ to be nonzero for a equal number of $k_1$ and $k_2$
components. In this case the generator of the process is equal to
$$L_u =C \nu \Delta +\f{\partial}{\partial t} +\partial_u$$
with $2C=\sum_k \lambda_k^2$ (c.f.\cite{Cipriano-Cruzeiro:07} Theorem 2.2). We shall assume $C$ to be equal to one.
Let us take two Lagrangian stochastic trajectories starting from different diffeomorphisms and let us write 
\begin{equation}
\label{E68} 
dg_t =(odW(t))+u(t, g_t )dt,\qquad d\tilde g_t =(odW(t))+u(t, \tilde g_t )dt
\end{equation}
with 

$$g_0 =\phi ,\qquad\tilde g_0 =\psi, \qquad \phi\not=\psi$$
We consider the $L^2$ distance of the particles defined by
$$\rho^2 (\phi ,\psi )= \int_{\mathbb T} |\phi (\theta )-\psi (\theta )|^2 \,d\theta .$$
where $d\theta$ stands for the normalized Lebesgue measure on the torus.

We let $\rho_t=\rho (g_t, \tilde g_t)$ and  $\tau (g , \tilde g )= \inf \{ t>0:  \rho_t =0\}$.

\begin{lemma}
 \label{L1}
The stopping time $\tau (g , \tilde g )$ is  infinite.
\end{lemma}
\begin{proof}  
By uniqueness of the solution of the sde for $\tilde g_t$ we can let for all $t>0$ $\tilde g_t(\theta )=g_t((\phi^{-1}\circ \psi)(\theta ))$. Since $g_t$, $\varphi$ and $\psi$ 
are diffeomorphisms, if $\varphi(\theta )\not=\psi(\theta )$ then 
$ g_t(\theta )\not=g_t((\phi^{-1}\circ \psi)(\theta ))$.

 Since $\phi \neq \psi $, the set $\di \left\{\theta\in \TT,\  \tilde g_t(\theta)\not=g_t(\theta)\right\}$ has positive measure and this implies that $\rho_t>0$, which in turn implies that $\tau (g, \tilde g)$ is infinite.
\end{proof}

Denote by $L_t(\theta)$ the local time of the process $|g_t(\theta)-\tilde g_t(\theta)|$ when $(g_t(\theta),\tilde g_t(\theta))$ reaches the cutlocus of $\TT$.
By It\^o calculus we have
\begin{align*}
 d\rho_t=&\f1{\rho_t}\sum_k \l_k \sqrt \nu \left\langle g_t-\tilde g_t, \left(A_k(g_t)-A_k(\tilde g_t)\right)dx_k(t)+\left(B_k(g_t)-B_k(\tilde g_t)\right)dy_k(t)\right\rangle_\TT
\\&+\f1{\rho_t}\left\langle g_t-\tilde g_t, u(t,g_t)-u(t,\tilde g_t)\right\rangle_\TT\,dt -\f1{\rho_t}\int_\TT|g_t-\tilde g_t|(\theta)dL_t(\theta)\\
&+\f1{2\rho_t}\sum_k \l_k^2  \nu\left(\left\|A_k(g_t)-A_k(\tilde g_t)\right\|_\TT^2+\left\|B_k(g_t)-B_k(\tilde g_t)\right\|_\TT^2\right)\,dt\\
&-\f1{2\rho_t^3}\sum_k \l_k^2  \nu\left(\left\langle g_t-\tilde g_t, A_k(g_t)-A_k(\tilde g_t)\right\rangle_\TT^2+\left\langle g_t-\tilde g_t, B_k(g_t)-B_k(\tilde g_t)\right\rangle_\TT^2\right)\,dt
\end{align*}
where $\langle \cdot,\cdot\rangle_\TT$ and $\|\cdot\|_\TT$ denote, resp., the $L^2$ inner product and norm. We let 
\begin{equation}
 \label{E70}
\d u(t)=\f1{\rho_t}\left(u(t,g_t)-u(t,\tilde g_t)\right).
\end{equation}
We have 
\begin{equation}
 \label{E72}
\ A_k(g_t)-A_k(\tilde g_t)=-2\sin\f{k\cdot (g_t+\tilde g_t)}{2}\sin\left(\f{k\cdot (g_t-\tilde g_t)}{2}\right)k^{\bot},
\end{equation}
\begin{equation}
 \label{E73}
\ B_k(g_t)-B_k(\tilde g_t)=2\cos\f{k\cdot (g_t+\tilde g_t)}{2}\sin\left(\f{k\cdot (g_t-\tilde g_t)}{2}\right)k^{\bot},
\end{equation}
where we have noted $k^{\bot} =(k_2 , -k_1 )$.
Then, for $k\neq 0$ we let 
\begin{equation}
 \label{E74}
n_k=\f{k}{|k|}, \qquad \hbox{and}\qquad n_g(t)=\f1{\rho_t}(g_t-\tilde g_t).
\end{equation}
This yields
\begin{equation}
 \label{E72a}
\ A_k(g_t)-A_k(\tilde g_t)=-2|k|^2\rho_t \sin\f{k\cdot (g_t+\tilde g_t)}{2}\f{\sin}{|k|\rho_t}\left(\f{k\cdot (g_t-\tilde g_t)}{2}\right)n_{k^{\bot}},
\end{equation}
\begin{equation}
 \label{E72b}
\ B_k(g_t)-B_k(\tilde g_t)=2|k|^2\rho_t (\cos\f{k\cdot (g_t+\tilde g_t)}{2}\f{\sin}{|k|\rho_t}\left(\f{k\cdot (g_t-\tilde g_t)}{2}\right)n_{k^{\bot}}.
\end{equation}
With these notations we get
\begin{align*}
 &d\rho_t\\&=\rho_t\sqrt \nu\sum_k \l_k |k|^2  \int_\TT2\left(n_{k^\bot}\cdot n_g(t,\theta)\right)\f{\sin}{|k|\rho_t}\left(\f{k\cdot (g_t(\theta)-\tilde g_t(\theta))}{2}\right)\,\\&\times\left(-\sin\f{k\cdot (g_t(\theta)+\tilde g_t(\theta))}{2}dx_k(t)+\cos\f{k\cdot (g_t(\theta)+\tilde g_t(\theta))}{2} dy_k(t)\right)\, d\theta\\
&+\rho_t\left\langle n_g(t), \d u(t)\right\rangle_\TT\,dt-\rho_t\int_\TT|n_g(t,\theta)|\f1{\rho_t}dL_t(\theta)\\
&+2\nu\rho_t\sum_k \l_k^2 |k|^4  \left\|\f{\sin}{|k|\rho_t}\left(\f{k\cdot (g_t-\tilde g_t)}{2}\right)\right\|^2_{\TT}\,dt\\
&-2\nu\rho_t\sum_k\l_k^2 |k|^4 \\&\times\left(\int_\TT\left(n_{k^\bot}\cdot n_g(t,\theta)\right)\sin\left(\f{k\cdot (g_t(\theta)+\tilde g_t(\theta))}{2}\right)\f{\sin}{|k|\rho_t}\left(\f{k\cdot (g_t(\theta)-\tilde g_t(\theta))}{2}\right)
\,d\theta\right)^2\,dt\\
&-2\nu\rho_t\sum_k\l_k^2 |k|^4 \\&\times\left(\int_\TT\left(n_{k^\bot}\cdot n_g(t,\theta)\right)\cos\left(\f{k\cdot (g_t(\theta)+\tilde g_t(\theta))}{2}\right)\f{\sin}{|k|\rho_t}\left(\f{k\cdot (g_t(\theta)-\tilde g_t(\theta))}{2}\right)
\,d\theta\right)^2\,dt.
\end{align*}
And finally:
\begin{prop}
\label{P2}
The It\^o equation for the distance $\rho_t$ between the diffeomorphisms $g_t$ and $\tilde g_t$ is given by
\begin{equation}
 \label{E76}
d\rho_t=\rho_t\left(\s_tdz_t+b_t\,dt+\left\langle n_g(t), \d u(t)\right\rangle_\TT\,dt -da_t \right)
\end{equation}
where $z_t$ is a real valued Brownian motion, $\s_t>0$ is given by
\begin{equation}
 \label{E78}
\begin{split}
\s_t^2=&4\nu\sum_k\l_k^2 |k|^4 \\&\times\left(\int_\TT\left(n_{k^\bot}\cdot n_g(t,\theta)\right)\sin\left(\f{k\cdot (g_t(\theta)+\tilde g_t(\theta))}{2}\right)\f{\sin}{|k|\rho_t}\left(\f{k\cdot (g_t(\theta)-\tilde g_t(\theta))}{2}\right)
\,d\theta\right)^2\\
&+4\nu\sum_k\l_k^2 |k|^4 \\&\times\left(\int_\TT\left(n_{k^\bot}\cdot n_g(t,\theta)\right)\cos\left(\f{k\cdot (g_t(\theta)+\tilde g_t(\theta))}{2}\right)\f{\sin}{|k|\rho_t}\left(\f{k\cdot (g_t(\theta)-\tilde g_t(\theta))}{2}\right)
\,d\theta\right)^2,
\end{split}
\end{equation}
the process $b_t$ satisfies
\begin{equation}
 \label{E78.1}
\begin{split}
b_t+\f12\s_t^2=&2\nu\rho_t\sum_k \l_k^2 |k|^4  \left\|\f{\sin}{|k|\rho_t}\left(\f{k\cdot (g_t-\tilde g_t)}{2}\right)\right\|^2_{\TT}\,dt
\end{split}
\end{equation}
   and $a_t$ is defined by
\begin{equation}
 \label{E78.2}
\begin{split}
a_0=0,\qquad da_t=&\int_\TT|n_g(t,\theta)|\f1{\rho_t}dL_t(\theta).
\end{split}
\end{equation}
\end{prop}
 So we have for all $0<t_0<t$, 
\begin{align}
 \label{E77}
\rho_t=\rho_{t_0}\exp\left(\int_{t_0}^t\s_s\,dz_s+\int_{t_0}^t\left
(b_s-\f12\s_s^2+\left\langle n_g(s), \d u(s)\right\rangle_\TT\right)\,ds- (a_t- a_{t_0} )\right).
\end{align}
Let 
\begin{equation}\label{deltak}
\d_k=\d_k(t,\theta)=\f{\rho_t(n_g\cdot n_k)}{|g_t(\theta)-\tilde g_t\theta)|}.
\end{equation}
Notice that
$$
\d_k^2+\d_{k^\bot}^2=1.
$$
\begin{lemma}
 \label{L2}
We have 
\begin{equation}
\label{boundsigma}
 \begin{split}
  \s_t^2
&\le4\nu\sum_k \l_k^2 |k|^4\int_\TT \d_{k^\bot}^2\f{\sin^2}{|k|^2\rho_t^2}\left(\f{k\cdot (g_t(\theta)-\tilde g_t(\theta))}{2}\right)d\theta
 \end{split}
\end{equation}
and 
\begin{equation}
 \label{boundbt}
b_t\ge 2\nu\sum_k \l_k^2 |k|^4 \int_\TT(n_g\cdot n_{k})^2\,d\theta\int_\TT \f{\sin^2}{|k|^2\rho_t^2}\left(\f{k\cdot (g_t(\theta)-\tilde g_t(\theta))}{2}\right)d\theta,
\end{equation}
in particular $b_t\ge 0$.

Let $R>0$.
Assuming  that 
  $\l_k=0$ for all $k$ such that   $|k|> R$ then on 
$$
\left\{\om\ |\  \forall\theta\in \TT,\ |g_t(\theta)-\tilde g_t(\theta)|\le\f{\pi}{R}\right\}
$$
we have 
\begin{align*}
b_t-\f12\s_t^2&\ge 0.
\end{align*}
\end{lemma}
\begin{proof}
 Using Cauchy Schwartz inequality, 
\begin{align*}
\s_t^2& \le 4\sum_k \l_k^2 |k|^4 \nu \int_\TT |(n_g\cdot n_{k^\bot})|\sin^2 \left(\f{k\cdot (g_t(\theta)+\tilde g_t(\theta))}{2}\right)\f{|\sin|}{|k|\rho_t}\left(\f{k\cdot (g_t(\theta)-\tilde g_t(\theta))}{2}\right) d\theta\\&\times \int_\TT |(n_g\cdot n_{k^\bot})|)\f{|\sin|}{|k|\rho_t}\left(\f{k\cdot (g_t(\theta)-\tilde g_t(\theta))}{2}\right)d\theta\\
&+ 4\sum_k \l_k^2 |k|^4 \nu \int_\TT   |(n_g\cdot n_{k^\bot})|\cos^2 \left(\f{k\cdot (g_t(\theta)+\tilde g_t(\theta))}{2}\right)\f{|\sin|}{|k|\rho_t}\left(\f{k\cdot (g_t(\theta)-\tilde g_t(\theta))}{2}\right)d\theta \\&\times\int_\TT |(n_g\cdot n_{k^\bot})|\f{|\sin|}{|k|\rho_t}\left(\f{k\cdot (g_t(\theta)-\tilde g_t(\theta))}{2}\right)d\theta\\
&=4\nu\sum_k \l_k^2 |k|^4 \left(\int_\TT |(n_g\cdot n_{k^\bot})|\f{|\sin|}{|k|\rho_t}\left(\f{k\cdot (g_t(\theta)-\tilde g_t(\theta))}{2}\right)d\theta\right)^2\\
&=4\nu\sum_k \l_k^2 |k|^4\left(\int_\TT \f{\d_{k^\bot}|g_t(\theta)-\tilde g_t\theta)|}{\rho_t}\f{|\sin|}{|k|\rho_t}\left(\f{k\cdot (g_t(\theta)-\tilde g_t(\theta))}{2}\right)d\theta\right)^2\\
&\le 4\nu\sum_k \l_k^2 |k|^4\int_\TT\f{|g_t(\theta)-\tilde g_t(\theta)|^2}{\rho_t^2}\,d\theta\int_\TT \d_{k^\bot}^2\f{\sin^2}{|k|^2\rho_t^2}\left(\f{k\cdot (g_t(\theta)-\tilde g_t(\theta))}{2}\right)d\theta\\
&=4\nu\sum_k \l_k^2 |k|^4\int_\TT \d_{k^\bot}^2\f{\sin^2}{|k|^2\rho_t^2}\left(\f{k\cdot (g_t(\theta)-\tilde g_t(\theta))}{2}\right)d\theta.
\end{align*}
On the other hand, 
\begin{align*}
 &b_t+\f12\s_t^2= 2\nu\sum_k \l_k^2 |k|^4 \int_\TT \f{\sin^2}{|k|^2\rho_t^2}\left(\f{k\cdot (g_t(\theta)-\tilde g_t(\theta))}{2}\right)d\theta,
\end{align*}
so that using the bound 
\begin{align*}
\s_t^2&\le 4\nu\sum_k \l_k^2 |k|^4 \left(\int_\TT |(n_g\cdot n_{k^\bot})|\f{|\sin|}{|k|\rho_t}\left(\f{k\cdot (g_t(\theta)-\tilde g_t(\theta))}{2}\right)d\theta\right)^2\\
&\le 4\nu\sum_k \l_k^2 |k|^4 \int_\TT(n_g\cdot n_{k^\bot})^2\,d\theta \int_\TT \f{\sin^2}{|k|^2\rho_t^2}\left(\f{k\cdot (g_t(\theta)-\tilde g_t(\theta))}{2}\right)d\theta
\end{align*}
for $\s_t^2$ yields 
\begin{align*}
 b_t&\ge 2\nu\sum_k \l_k^2 |k|^4 \int_\TT(n_g\cdot n_{k})^2\,d\theta\int_\TT \f{\sin^2}{|k|^2\rho_t^2}\left(\f{k\cdot (g_t(\theta)-\tilde g_t(\theta))}{2}\right)d\theta
\end{align*}
where we used 
$$
\int_\TT(n_g\cdot n_{k})^2\,d\theta+\int_\TT(n_g\cdot n_{k^\bot})^2\,d\theta=1.
$$
Since $\l_k$ depends only on $|k|$, we have $\l_k=\l_{k^\bot}$ for all $k$. Then putting together the terms corresponding to $k$ and $k^\bot$ we obtain
\begin{align*}
 &b_t+\f12\s_t^2= \nu\sum_k \l_k^2 |k|^4\\
&\times \int_\TT \left(\f{\sin^2}{|k|^2\rho_t^2}\left(\f{k\cdot (g_t(\theta)-\tilde g_t(\theta))}{2}\right)+\f{\sin^2}{|k|^2\rho_t^2}\left(\f{k^{\bot}\cdot (g_t(\theta)-\tilde g_t(\theta))}{2}\right)\right)d\theta,
\end{align*}
and this yields using the bound for $\s_t^2$ as well as $\d_k^2+\d_{k^\bot}^2=1$ 
\begin{align*}
& b_t-\f12\s_t^2\ge\nu \sum_k \l_k^2 |k|^4\\&
\times\int_\TT (\d_k^2-\d_{k^\bot}^2)\left(\f{\sin^2}{|k|^2\rho_t^2}\left(\f{k\cdot (g_t(\theta)-\tilde g_t(\theta))}{2}\right)-\f{\sin^2}{|k|^2\rho_t^2}\left(\f{k^{\bot}\cdot (g_t(\theta)-\tilde g_t(\theta))}{2}\right)\right)d\theta\\
&=\nu \sum_k \l_k^2 |k|^4\\&
\times\int_\TT (\d_k^2-\d_{k^\bot}^2)\left(\f{\sin^2}{|k|^2\rho_t^2}\left(\d_k\f{|k||g_t-\tilde g_t|(\theta)}{2}\right)-\f{\sin^2}{|k|^2\rho_t^2}\left(\d_{k^\bot}\f{|k||g_t-\tilde g_t|(\theta)}{2}\right)\right)d\theta
\end{align*}

Assuming that   $\l_k=0$ whenever $|k|> R$ then on
$$
\left\{\om| \forall\theta\in \TT,\ |g_t(\theta)-\tilde g_t(\theta)|\le\f{\pi}{R}\right\}
$$
the functions inside the integral are nonegative, consequently
\begin{align*}
b_t-\f12\s_t^2&\ge 0.
\end{align*}
\end{proof}
Define
\begin{equation}
 \label{E71}
\ell(x)=\f{\sin x}{x}\ \hbox{for }\ x\not=0,\quad \ell(0)=1.
\end{equation}
From Lemma~\ref{L2} we easily get the following result.

\begin{prop}
\label{P4}
Let $R\ge 1$. Then on 
$$
\left\{\om| \forall\theta\in \TT,\ |g_t(\theta)-\tilde g_t(\theta)|\le\f{\pi\sqrt2}{R}\right\},
$$
 letting  $$ c_R=\f{\nu}{8}\ell^2\left(\f{\pi}{\sqrt{2}}\right)\sum_{|k|\le R}\l_k^2|k|^4,$$
we have,
\begin{equation}
\label{EF1}
\begin{split}
&d\rho_t\ge \rho_t \left(\s_t dz_t -\|\d u(t)\|_\TT\,dt-\int_\TT|n_g(t,\theta)|\f1{\rho_t}dL_t(\theta)+c_R\,dt\right).
\end{split}
\end{equation}
Moreover assuming that $\l_k=0$ whenever $|k|>R$, then 
letting
$$
c_R'=\f18\nu\inf_{|v|=1}\sum_{|k|\le R}\l_k^2|k|^4\left((n_k\cdot v)^2-(n_{k^\bot}\cdot v)^2\right)^2,
$$ 
on
$$
\left\{\om| \forall\theta\in \TT,\ |g_t(\theta)-\tilde g_t(\theta)|\le\f{\pi}{2R}\right\},
$$
\begin{equation}
\label{EF2}
\begin{split}
 &d\rho_t\ge \rho_t \left(\s_t dz_t +\f12\s_t^2\,dt -\|\d u(t)\|_\TT\,dt+c_R'\,dt\right).
\end{split}
\end{equation}
\end{prop}
\begin{proof}
If $\di |g_t(\theta)-\tilde g_t(\theta)|\le\f{\pi\sqrt2}{R}$ then for all $k$ such that $|k|\le R$, 
$$
\ell^2\left(\f{k\cdot (g_t(\theta)-\tilde g_t(\theta))}{2}\right)\ge \ell^2\left(\f{\pi}{\sqrt{2}}\right)
$$
and this implies 
$$
\f{\sin^2}{|k|^2\rho_t^2}\left(\f{k\cdot (g_t(\theta)-\tilde g_t(\theta))}{2}\right)\ge \f14\ell^2\left(\f{\pi}{\sqrt{2}}\right)(n_k\cdot n_g)^2.
$$
So with~\eqref{boundbt} we get 
\begin{align*}
 \label{boundbt}
b_t&\ge \f12\ell^2\left(\f{\pi}{\sqrt{2}}\right)\nu\sum_{|k|\le R} \l_k^2 |k|^4 \left(\int_\TT(n_g\cdot n_{k})^2\,d\theta\right)^2\\
&\ge \f14\ell^2\left(\f{\pi}{\sqrt{2}}\right)\nu\sum_{|k|\le R} \l_k^2 |k|^4 \left(\left(\int_\TT(n_g\cdot n_{k})^2\,d\theta\right)^2+\left(\int_\TT(n_g\cdot n_{k^\bot})^2\,d\theta\right)^2\right)\\
&\ge \f18\ell^2\left(\f{\pi}{\sqrt{2}}\right)\nu\sum_{|k|\le R} \l_k^2 |k|^4
\end{align*}
(again we used $
\int_\TT(n_g\cdot n_{k})^2\,d\theta+\int_\TT(n_g\cdot n_{k^\bot})^2\,d\theta=1.
$).
This establishes \eqref{EF1}.

Next if $\di |g_t(\theta)-\tilde g_t(\theta)|\le\f{\pi}{2R}$ then from the calculation in the proof of Lemma~\ref{L2}
\begin{align*}
 &b_t-\f12\s_t^2\ge\nu \sum_{|k|\le R} \l_k^2 |k|^4\\&
\times\int_\TT (\d_k^2-\d_{k^\bot}^2)\left(\f{\sin^2}{|k|^2\rho_t^2}\left(\d_k\f{|k||g_t-\tilde g_t|(\theta)}{2}\right)-\f{\sin^2}{|k|^2\rho_t^2}\left(\d_{k^\bot}\f{|k||g_t-\tilde g_t|(\theta)}{2}\right)\right)d\theta\\
&\ge \nu \sum_{|k|\le R} \l_k^2 |k|^4\int_\TT (\d_k^2-\d_{k^\bot}^2)^2\f{|g_t-\tilde g_t|^2(\theta)}{8\rho_t^2}\,d\theta\\
&\ge \int_\TT\f{|g_t-\tilde g_t|^2(\theta)}{\rho_t^2}c_R'\,d\theta=c_R'.
\end{align*}
this establishes \eqref{EF2}.
\end{proof}

\begin{thm}
 \label{T2}
Let $t>0$, $R\ge 1$  and 
$$
\Om_t=\left\{\om\in\Om,\ \forall s\le t, \ \forall \theta\in \TT,\ |(g_s(\theta)(\om)-\tilde g_s(\theta)(\om))|\le \f{\pi}{2R}\right\}.
$$
If we assume the initial conditions for the $L^2$ distance and the $L^2$ norm of the initial velocity related as $c=\rho_0 -2\|u_0 \|_\TT >0$, and suppose that $\int_{\TT} u =0$, 
 then on $\Om_t$, 
\begin{equation}\label{EF5}
 \forall s\le t,\ \ \rho_s\ge e^{\int_0^t\s_s\,dz_s+c_R't}\left(\rho_0-2\|u_0\|_\TT\int_0^te^{-\int_0^s\s_r\,dz_r-(c_R'+\f{\nu}2)s}\,ds\right)
\end{equation}
as long as the right hand side stays positive.

On the other hand if we assume that there exist constants $c_1,c_2>0$ such that for all $\theta\in \TT$ and $s\in [0,t]$ , 
\begin{equation}
\label{boundgradu}
|\n u(t,\theta)|\le c_1e^{-c_2t},
\end{equation} then on $\Om_t$, $\forall s\le t,\ \ $
\begin{equation}\label{EF5bis}
 \rho_s\ge \rho_{0}\exp\left(\int_{0}^t\s_s\,dz_s+c_R't-\f{c_1}{c_2}\left(1-e^{-c_2t}\right)\right).
\end{equation}
\end{thm}
 
\begin{proof} Assume that $\rho_0 -2\|u_0 \|_\TT >0$.
From inequality~\eqref{EF2} we have on $\Om_t$, for $s\le t$,
\begin{equation}
\label{EF6}
\begin{split}
 &d\rho_s\ge \rho_s \left(\s_s dz_s + (c_R'+\f12\s_s^2)\,ds\right)-2\|u(s,\cdot)\|_\TT\,ds.
\end{split}
\end{equation}
Using the fact that $u(t,.)$ satisfies Navier-Stokes equation together with Poincar\'e inequality,
\begin{align*}\f{d}{ds}||u(s,.)||_\TT^2 &=-2\nu  ||\nabla u (s ,.)||_\TT^2 \\
&\le  - \nu  ||u(s ,.)||_\TT^2.
\end{align*}

Therefore we have

$$||u (s,.)||_\TT \le e^{-\f{\nu}{2} s} ||u_0||_\TT.$$
We obtain
\begin{equation}
\label{EF7}
\begin{split}
d\rho_s\ge \rho_s \left(\s_s dz_s + (c_R'+\f12\s_s^2)\,ds\right)-2e^{-\f{\nu}{2} s} ||u_0||_\TT\,ds.
\end{split}
\end{equation}
 From this comparison theorem for solution of sde's yields~\eqref{EF5}.
 
 Now assume~\eqref{boundgradu}. To prove~\eqref{EF5bis} we start with~\eqref{EF2}, and  remark that $\di \|\d u(t)\|_\TT\le \sup_{\theta\in\TT}|\n u(t,\theta)|$.
 Then with the bound on $\n u(t,\theta)$ we have
 $$
 d\rho_t\ge \rho_t \left(\s_t dz_t +\f12\s_t^2\,dt -c_1e^{-c_2t}\,dt+c_R'\,dt\right).
 $$
 Integrating the right hand side between $t_0$ and $t$ gives the result.
\end{proof}
\begin{remark}
 The bound~\eqref{boundgradu} is satisfied for instance for solutions $u(t,\cdot)$ of the form $e^{-\nu |k|^2 t } A_k$. 

Also notice that, by the expression of the constant $c_R'$, the stochastic Lagrangian trajectories for a fluid with a given viscosity constant tend to get apart faster when the higher Fourier modes (and therefore the smaller lenght scales) are randomly excited.
\end{remark}

\section{The two-dimensional torus endowed with the extrinsic distance}\label{Section7}
\setcounter{equation}0

It seems difficult to deal with the local time term of Proposition~\ref{P2}. To circumvent this problem we propose to endow the torus $\TT$ with a distance $\rho_\TT$ equivalent to the  one of section~\ref{Section6}, but such that $\rho_\TT^2$ is smooth on $\TT\times \TT$. Then we will see that when the assumptions of Theorem~\ref{T2} are not fulfilled, then the behaviour of the distance of two diffeomorphisms can be completely different even if their distance is small. So the uniform control of the distance in Theorem~\ref{T2} looks as a necessary condition for an exponential growth of the distance.

The map
\begin{align*}
 \RR/2\pi\ZZ\times \RR/2\pi\ZZ&\to [0,2]\\
(\theta_1,\theta_2)&\mapsto 2\left|\sin\left(\f{\theta_2-\theta_1}{2}\right)\right|
\end{align*}
defines a distance on the circle $\RR/2\pi\ZZ$: it is the extrinsic distance on the circle embedded in the plane. From this distance we can define the product distance on the torus $\TT$.
$$
\rho_\TT((\theta_1,\theta_2),(\theta_1',\theta_2'))=2\left(\sin^2\left(\f{\theta_1'-\theta_1}{2}\right)+\sin^2\left(\f{\theta_2'-\theta_2}{2}\right)\right)^{1/2}.
$$
Note 
$$
\rho_\TT^2((\theta_1,\theta_2),(\theta_1',\theta_2'))=2\left(2-\cos(\theta_1'-\theta_1)-\cos(\theta_2'-\theta_2)\right).
$$
The distance $\rho_\TT^2$ is smooth on $\TT\times \TT$.
Now let $\phi$ and $\psi$ be two diffeomorphisms on the torus $\TT$. We define the distance $\rho(\phi,\psi)$ with the formula
\begin{align*}
 \rho^2(\phi,\psi)&=\int_\TT\rho_\TT^2(\phi(\theta), \psi(\theta))\,d\theta\\
&=2\int_\TT\left(2-\cos (\phi^1(\theta)-\psi^1(\theta))-\cos (\phi^2(\theta)-\psi^2(\theta))\right)\,d\theta\\
&=4\int_\TT\left(\sin^2\left(\f{\phi^1(\theta)-\psi^1(\theta)}{2}\right)+\sin^2\left(\f{\phi^2(\theta)-\psi^2(\theta)}{2}\right)\right)\,d\theta
\end{align*}
Now let 
$$
\rho_t=\rho(g_t,\tilde g_t).
$$
>From the smoothness of $\rho_\TT^2$, the formula for $\rho_t$ does not involve a local time.
 More precisely, 
letting 
$$
\d g=g_t(\theta)-\tilde g_t(\theta),$$
$$
\d \cos k\cdot g=\cos k\cdot g_t(\theta)-\cos k\cdot \tilde g_t(\theta),$$
$$
 \d \sin k\cdot g=\sin k\cdot g_t(\theta)-\sin k\cdot \tilde g_t(\theta), 
$$
$$
\sin\d g=(\sin (\d g_t)_1(\theta), \ \sin(\d g_t)_2(\theta)),
$$
$$
\d u =\left(u(t,g_t)-u(t,\tilde g_t)\right)
$$
we get from It\^o calculus 
\begin{align*}
d\rho_t&=\rho_t\sum_k \lambda_k \left\langle \f{\sin\d g}{\rho_t}, (k_2,-k_1)\left(\f{\d \cos k\cdot g}{\rho_t}dx_k+\f{\d \sin k\cdot g}{\rho_t}dy_k\right)\right\rangle_\TT \\
&+ \rho_t\left\langle \f{\sin\d g}{\rho_t},\ \f{\d u}{\rho_t}  \right\rangle_\TT\,dt\\
&+\f{\rho_t}{2}\left(\sum_k  \lambda_k^2 \int_\TT\left(k_2^2\cos \d g_1+k_1^2\cos\d g_2\right)\f{(\d \cos k\cdot g)^2+(\d \sin k\cdot g)^2}{\rho_t^2}\,d\theta\right)\,dt\\
&-\f{\rho_t}{2}\sum_k \lambda_k^2\left(\int_\TT\left(k_2\f{\sin\d g_1}{\rho_t}-k_1\f{\sin\d g_2}{\rho_t}\right)\f{\d \cos k\cdot g}{\rho_t}\,d\theta\right)^2\,dt\\
&-\f{\rho_t}{2}\sum_k \lambda_k^2 \left(\int_\TT\left(k_2\f{\sin\d g_1}{\rho_t}-k_1\f{\sin\d g_2}{\rho_t}\right)\f{\d \sin k\cdot g}{\rho_t}\,d\theta\right)^2\,dt.
\end{align*}
This clearly has the form
$$
d\rho_t=\rho_t\left(\s_t\,dz_t+b_t\,dt\right)
$$
where $\s_t$ and $b_t$ are bounded processes and $z_t$ is a real- valued Brownian motion. However it can happen that the drift is negative even if $\rho_t$ is small, as the following example shows. 
\begin{example}
Let $\a>0$ small and $\e>0$ satisfying $\e<<\a$. Take $\phi=\id$ and assume that there exist two subsets $E_1$ and $E_2$ of $\TT$ such that $E_1\subset E_2$, $E_1$ has measure $\a$, $E_2$ has measure $\a+\e$, $\psi(\theta)=\theta$ for all $\theta \in \TT\backslash E_2$ and $\psi(\theta)=(\theta_1+\pi,\theta_2)$ for all $\theta\in E_1$. Since $\e$ can be as small as we want, we have 
$$
\rho_0^2\simeq 4\a,\quad (\sin\d g)_0\simeq 0,\quad (\d g_0)_2\simeq 0,\quad (\d \sin k\cdot g)_0\simeq 0,
$$
$$
\hbox{on}\  \TT\backslash E_2,
\quad 
(\d \cos k\cdot g)_0=0,
$$
$$
\hbox{
on}\ E_1,
\quad
(\d \cos k\cdot g)_0= -2\quad \hbox{if $k_1$is odd},\quad (\d \cos k\cdot g)_0= 0\quad \hbox{if $k_1$is even},
$$
so at time $t=0$, 
\begin{align*}
d\rho_t\simeq
-\f{\rho_t}{2}\left(\sum_{ k_1\ \hbox{odd}} \lambda_k^2 k_2^2 \right)\,dt.
\end{align*}

 To construct a diffeomorphism like $\psi$, one can cut an annulus $E_1$ of width $\di\f{\a}{2\pi}$ in $\TT$ and rotate it by $\pi$. This yields a  one to one map on $\TT$. Then smoothen it around the boundary of the annulus  to get $\psi$. The set $E_2$ can be taken as an annulus of width $\di\f{\a+\e}{2\pi}$ containing $E_1$.
\end{example}

\section{Distance  and rotation processes of two particles on a general Riemannian manifold}\label{Section3bis} 
\vskip 3mm

\subsection{Distance of two particles}
$ $

\vskip 2mm 

 Let $B_t=(B_t^\ell)_{\ell\ge 0}$ be a family of independent real Brownian motions, $\s=(\s_\ell)_{\ell\ge 0}$, with,  for all $\ell\ge 0$, $\s_\ell$ a divergence free vector field on $M$. We furthermore assume that 
\begin{equation}
 \label{E94}
\s(x)\s^\ast(y)=a(x,y).
\end{equation}
In particular 
\begin{equation}
 \label{E93}
\s(x)\s^\ast(x)=2\nu {\bf g}^{-1}(x).
\end{equation}
We let $\varphi, \psi\in G_V^0$.
In this section we assume that 
\begin{equation}
\label{E90}
dg_t(x)=\s(g_t(x))\,dB_t+u(t,g_t(x))\,dt, \quad g_0=\varphi
\end{equation}
and 
\begin{equation}
\label{E91}
d\tilde g_t(x)=\s(\tilde g_t(x))\,dB_t+u(t,\tilde g_t(x))\,dt, \quad \tilde g_0=\psi
\end{equation}

For simplicity we let $x_t=g_t(x)$, $y_t=\tilde g_t(x)$ and  $$\rho_t(x)=\rho_M(x_t,y_t)$$
 For  $x,y\in M$ such that $y$ does not belong to the cutlocus of $x$, we let  $a\mapsto \g_a(x,y)$ be the minimal geodesic in time $1$ from $x$ to $y$ ($\g_0(x,y)=x$, $\g_1(x,y)=y)$). For $a\in [0,1]$ we let $J_a=T\g_a$ the tangent map to $\g_a$. In other words, for $v\in T_xM$ and $w\in T_yM$, $J_a(v,w)$ is the value at time $a$ of the Jacobi field along $\g_\cdot$ which takes the values $v$ at time $0$ and $w$ at time $1$.

We first consider the case where $y_t$ does not belong to the cutlocus of $x_t$.
 We note $T_a=T_a(t)=\dot \g_a(x_t,y_t)$ and $\g_a(t)=\g_a(x_t,y_t)$.


Letting $P(\g_a)_t$ be the parallel transport along $\g_a(t)$, we have for the It\^o covariant differential
\begin{align*}
\SD\dot\g_a(t)&:=P(\g_a)_td\left(P(\g_a)_t^{-1}\dot\g_a(t)\right)\\
&=\n_{d\g_a(t)}\dot\g_a+\f12 \n_{ d\g_a(t)}\cdot \n_{ d\g_a(t)}\dot\g_a (t).
\end{align*}
On the other hand the It\^o differential $d\g_a(t)$ satisfies
\begin{align*}
 d\g_a(t)&= J_a(dx_t,dy_t)+\f12\left(\n_{(dx_t,dy_t)}J_a\right)(dx_t,dy_t).
\end{align*}
So we get 
\begin{equation}
 \label{E80}
 \SD\dot\g_a(t)=\n_{J_a(dx_t,dy_t)}\dot\g_a+\n_{\f12\left(\n_{(dx_t,dy_t)}J_a\right)(dx_t,dy_t)}\dot\g_a+\f12 \n_{ d\g_a(t)}\cdot \n_{ d\g_a(t)}\dot\g_a (t).
\end{equation}
Let $e(t)\in T_{x_t}M$  be the unit vector satisfying $T_0(t)=\rho_t(x)e(t)$.  
 For $\ell\ge 0$ we let $a\mapsto J_a^\ell(t,x)$ be the Jacobi field such that $J_0^\ell(t,x)=\s_\ell(g_t(x))$, $J_1^\ell(t)=\s_\ell(\tilde g_t(x))$. Moreover we assume that $\n_{J_0^\ell(t,x)}J_0^\ell(t,x)=0$ and  $\n_{J_1^\ell(t,x)}J_1
 ^\ell(t,x)=0$.

With these notations, equation~\eqref{E80}  rewrites as
\begin{align*}
\SD T_a&=\n_{J_a(dx_t,dy_t)}T_a+\f12\sum_{\ell \ge 0}\n_{\n_{J_a^\ell }J_a^\ell }T_a\,dt+\f12\sum_{\ell \ge 0}\n_{J_a^\ell }\cdot\n_{J_a^\ell } T_a\,dt\\
&=\dot J_a(dx_t,dy_t)+\f12 \sum_{\ell \ge 0}\n_{J_a^\ell }\n_{J_a^\ell } T_a\,dt.
\end{align*}
 We have
\begin{align*}
 d\rho_t(x)&= d\left(\left( \int_0^1\left\langle T_a(t),T_a(t)\right\rangle\, da\right)^{1/2}\right)\\
&=\f1{2\rho_t(x)}\left(2\int_0^1\left\langle \SD T_a(t),T_a(t)\right\rangle\, da+\int_0^1\left\langle \SD T_a(t),\SD T_a(t)\right\rangle\, da\right)\\
&-\f1{8\rho_t(x)^3}d\left(\|T_0\|^2\right)\cdot d\left(\|T_0\|^2\right)\\
&=\sum_{\ell\ge 0}\left\langle \dot J_0^\ell(t,x), e_t(x)  \right\rangle \,dB_t^\ell+\left\langle \dot J_0(u(t,g_t(x)), u(t,\tilde g_t(x))), e_t(x)\right\rangle\\
&+\f1{2\rho_t(x)}\left(\int_0^1 \sum_{\ell \ge 0 }\left\langle \n_{J_a^\ell }\n_{J_a^\ell }T_a,T_a\right\rangle\,da\,dt +\sum_{\ell \ge 0}\int_0^1\|\dot J_a^\ell \|^2\,da\right)\\
&-\f1{2\rho_t(x)}\sum_{\ell\ge 0}\langle \dot J_0^\ell(t,x),e_t(x)\rangle^2.
\end{align*}
Note
\begin{align*}
 \int_0^1\left\langle \n_{J_a^\ell }\n_{J_a^\ell }T_a,T_a\right\rangle\,da&=\int_0^1\left\langle \n_{J_a^\ell }\n_{T_a}J_a^\ell ,T_a \right\rangle\,da\\
&=\int_0^1\left\langle \n_{T_a}\n_{J_a^\ell }J_a^\ell ,T_a \right\rangle\,da-\int_0^1\left\langle R(T_a,J_a^\ell )J_a^\ell ,T_a \right\rangle\,da\\
&=\int_0^1T_a\left\langle \n_{J_a^\ell }J_a^\ell ,T_a \right\rangle\,da-\int_0^1\left\langle R(T_a,J_a^\ell )J_a^\ell ,T_a \right\rangle\,da\\
&=\left[\left\langle \n_{J_a^\ell }J_a^\ell ,T_a \right\rangle\right]_0^1-\int_0^1\left\langle R(T_a,J_a^\ell )J_a^\ell ,T_a \right\rangle\,da\\
&=-\int_0^1\left\langle R(T_a,J_a^\ell )J_a^\ell ,T_a \right\rangle\,da
\end{align*}
using the fact that $\n_{J_a^\ell}J_a^\ell=0$ for $a=0,1$.
So finally, 
\begin{align*}
d\rho_t(x)&=\sum_{\ell\ge 0}\left\langle \dot J_0^\ell(t,x), e_t(x)  \right\rangle \,dB_t^\ell\\
&+\left\langle \dot J_0(u(t,g_t(x)), u(t,\tilde g_t(x))), e_t(x)\right\rangle \\
&+ \f1{2\rho_t(x)}\left(\int_0^1\sum_{\ell\ge 0}\left(\|\dot J_a^{\ell, N}\|^2- 
\left\langle R(T_a(t,x),J_a^{\ell,N}(t,x))J_a^{\ell,N}(t,x),T_a(t,x)\right\rangle
\right)\,da\right)\,dt
\end{align*}
with $J_a^{\ell,N}(t,x)$ the part of $J_a^{\ell}(t,x)$ normal to $T_a$.

Removing the assumption that $y_t$ does not belong to the cutlocus of $x_t$, it is well known (see \cite{Kendall:86} for a similar argument) that  the formula becomes
\begin{align*}
d\rho_t(x)&=\sum_{\ell\ge 0}\left\langle \dot J_0^\ell(t,x), e_t(x)  \right\rangle \,dB_t^\ell\\
&+\left\langle \dot J_0(u(t,g_t(x)), u(t,\tilde g_t(x))), e_t(x)\right\rangle- dL_t(x) \\
&+ \f1{2\rho_t(x)}\left(\int_0^1\sum_{\ell\ge 0}\left(\|\dot J_a^{\ell, N}\|^2- 
\left\langle R(T_a(t,x),J_a^{\ell,N}(t,x))J_a^{\ell,N}(t,x),T_a(t,x)\right\rangle
\right)\,da\right)\,dt
\end{align*}
where $-L_t(x)$ is the local time of $\rho_t(x)$ when $(g_t(x), \tilde g_t(x))$ visits the cutlocus. Then letting
$$
\rho_t=\rho(g_t,\tilde g_t)=\left(\int_M\rho_t^2(x)\,dx\right)^{1/2}, 
$$
we get
\begin{align*}
d\rho_t&=\f1{\rho_t}\sum_{\ell\ge 0}\left(\int_M\rho_t(x)\left\langle\dot J_0^\ell(t,x), e_t(x)  \right\rangle\, dx\right)\, dB_t^\ell\\
&+\f1{\rho_t}\int_M\rho_t(x)\left\langle \dot J_0(u(g_t(x)), u(\tilde g_t(x))), e_t(x)\right\rangle\, dx\, dt-\f1{\rho_t}\int_M\rho_t(x)L_t(x)\,dx\\
&+\f1{2\rho_t}\left(\int_M\sum_{\ell\ge 0}\left(\int_0^1\left(\|\dot J_a^{\ell, N}\|^2- 
\left\langle R(T_a(t,x),J_a^{\ell,N}(t,x))J_a^{\ell,N}(t,x),T_a(t,x)\right\rangle
\right)\,da\right)\,dx\right)\,dt\\
&+\f1{2\rho_t}\int_M\sum_{\ell\ge 0}\left\langle\dot J_0^\ell(t,x), e_t(x)  \right\rangle^2\, dx\,dt\\&-\f1{2\rho_t^3}\sum_{\ell\ge 0}\left(\int_M\rho_t(x)\left\langle\dot J_0^\ell(t,x), e_t(x)  \right\rangle\,dx\right)^2\,dt.
\end{align*}
For a vector  $w\in T_{g_t(x)} M$,  we let $w^T$ the part of $w$ tangential to $T_0(t,x)$. 
Letting 
$$
\cos\left(\dot J_0^{\ell,T}(t,\cdot), T_0(t,\cdot)\right)=\f{\int_M\left\langle \dot J_0^{\ell,T}(t,x),T_0(t,x)\right\rangle \,dx}{\rho_t\left(\int_M\left\|\dot J_0^{\ell,T}(t,x)\right\|^2\, dx\right)^{1/2}}
$$
(observe $\di \rho_t^2=\int_M\left\|T_0(t,x)\right\|^2\, dx$), we finally proved 
\begin{prop}
\label{P3}
The It\^o differential of  the distance $\rho_t$ between $g_t$ and $\tilde g_t$ is given by
\begin{align*}
d\rho_t&=\f1{\rho_t}\sum_{\ell\ge 0}\left(\int_M\rho_t(x)\left(P_{\tilde g_t(x),g_t(x)}(\s_\ell^T(\tilde g_t(x)))-\s_\ell^T(g_t(x)\right)\, dx\right)\, dB_t^\ell\\
&+\f1{\rho_t}\int_M\rho_t(x)\left(P_{\tilde g_t(x),g_t(x)}(u^T(\tilde g_t(x))))-u^T(g_t(x))\right)\, dx\, dt-\f1{\rho_t}\int_M\rho_t(x)dL_t(x)\,dx\\
&+\f1{2\rho_t}\left(\int_M\sum_{\ell\ge 0}\left(\int_0^1\left(\|\dot J_a^{\ell, N}\|^2- 
\left\langle R(T_a(t,x),J_a^{\ell,N}(t,x))J_a^{\ell,N}(t,x),T_a(t,x)\right\rangle
\right)\,da\right)\,dx\right)\,dt\\
&+\f1{2\rho_t}\sum_{\ell\ge 0}\left(1-\cos^2\left(\dot J_0^{\ell,T}(t,\cdot), T_0(t,\cdot)\right)\right)\int_M\left\|\dot J_0^{\ell,T}(t,x)\right\|^2\, dx\,dt.
\end{align*}
\end{prop}

In the case of manifolds with negative curvature we may observe a similar phenomena
to the one of the torus with the Euclidean distance treated in Section 4: as long as the $L^\infty$ norm stays sufficiently small to avoid the cut-locus of the manifold,
the $L^2$ mean  distance between the stochastic particles tends to increase exponentially fast.

\vskip 10mm
\subsection{The rotation process}
$ $

\vskip 2mm

In the following we would like to study the rotation of two particles $g_t(x)$ and $\tilde g_t(x)$ when they are in a close distance one to another.
Recall that we have noted $x_t=g_t(x)$, $y_t=\tilde g_t(x)$.  We always assume that the distance from $x_t$ to $y_t$ is small: we are interested in the behaviour of $e(t)$ as $\rho_t(x)$ goes to~$0$.
We  let 
\begin{equation}
 \label{E82}
d_mx(t)^N=\s(x_t)dB_t-\langle \s(x_t)dB_t,e(t)\rangle e(t)
\end{equation}
and 
\begin{equation}
 \label{E82bis}
d_my(t)^N=\s(y_t)dB_t-\langle \s(y_t)dB_t,P_{x_t,y_t}e(t)\rangle P_{x_t,y_t}e(t)
\end{equation}
where $P_{x_t,\g_a(t)}$ denotes the parallel transport along $\g_a$.

>From It\^o formula we have 
\begin{equation}
 \label{R1}
\SD T_0=\rho_t(x)\SD e(t)+d\rho_t(x) e(t)+ d\rho_t(x)\SD e(t)
\end{equation}
and this yields 
\begin{align*}
\SD e(t)&=\f1{\rho_t(x)}\SD T_0-\f1{\rho_t(x)}d\rho_t(x) e(t)-\f12 \f1{\rho_t(x)}d\rho_t(x) \SD e(t)\\
&=\f1{\rho_t(x)}\dot J_0(d_mx(t)^N, d_my(t)^N)\\
&+ \f1{\rho_t(x)}\dot J_0(u(t,x_t),u(t,y_t))\,dt+\f1{2\rho_t(x)}\sum_{\ell\ge 0} \n_{J_0^\ell}\n_{J_0^\ell}T_0\,dt\\
&-\f1{\rho_t(x)}\left\langle P_{y_t,x_t}(u(t,y_t))-u(t,x_t),e(t)\right\rangle e(t)\\
&-\f1{2\rho_t(x)^2} \left(\int_0^1\sum_{\ell \ge 0}\left(\|\n_{T_a}J_a^\ell\|^2-R(T_a,J_a^\ell)J_a^\ell,T_a\right)\,da\right)  e(t)\\&
-\f12 \f1{\rho_t(x)}d\rho_t(x) \SD e(t)\\
&=\f1{\rho_t(x)}\dot J_0(d_mx(t)^N, d_my(t)^N)+\f1{\rho_t(x)}\dot J_0(u^N(t,x_t),u^N(t,y_t))\\
&+\f1{2\rho_t(x)}\sum_{\ell\ge 0} \n_{J_0^\ell}\n_{J_0^\ell}T_0\,dt\\
&-\f1{2\rho_t(x)^2} \left(\int_0^1\sum_{\ell\ge 0}\left(\|\n_{T_a}J_a^\ell\|^2-R(T_a,J_a^\ell)J_a^\ell,T_a\right)\,da\right)  e(t)
\end{align*}
where we used the fact that $d\rho_t(x) \SD e(t)=0$, and where $u^N$ denotes the part of $u$ which is normal to the geodesic $\g_a$.
Now as before 
\begin{align*}
\n_{J_0^\ell}\n_{J_0^\ell}T_0&=\n_{T_0}\n_{J_0^\ell}J_0^\ell-R(T_0, J_0^\ell)J_0^\ell.
\end{align*}
Finally we get
\begin{lemma}
\label{L5}
\begin{align*}
\SD e(t)&=\f1{\rho_t(x)}\dot J_0(d_mx(t)^N, d_my(t)^N)+\f1{\rho_t(x)}\dot J_0(u^N(t,x_t),u^N(t,y_t))\\
&+\f1{2\rho_t(x)}\sum_{\ell\ge 0}\n_{T_0}\n_{J^\ell}J^\ell-R(T_0, J_0^\ell)J_0^\ell\,dt\\
&-\f1{2\rho_t(x)^2} \left(\int_0^1\sum_{\ell\ge 0}\left(\|\n_{T_a}J_a^\ell\|^2-R(T_a,J_a^\ell)J_a^\ell,T_a\right)\,da\right)  e(t).
\end{align*}
\end{lemma}

>From now on we assume that $M=\TT$ the two dimensional torus. 

In this situation the curvature tensor vanishes and we have the formulas 
$$
J_a(v,w)=v+a(w-v),\qquad \dot J_a(v,w)=w-v.
$$
We immediately get 
\begin{align*}
de(t)=\SD e(t)=&\f1{\rho_t(x)}\left(d_my(t)^N- d_mx(t)^N\right)+\f1{\rho_t(x)}\left((u^N(t,y_t)-u^N(t,x_t)\right)\,dt\\&-\f1{2\rho_t(x)^2}\sum_{\ell\ge 0}\left\|\s_\ell(y_t)-\s_\ell(x_t)\right\|^2\,dt\, e(t)
\end{align*}
where we used the fact that $\n_{T_0}\n_{J^\ell}J^\ell=0$, as a consequence of $\n_{J_0^\ell}J_0^\ell=0$, $\n_{J_1^\ell}J_1^\ell=0$, and $R\equiv 0$.

Let us specialize again to the case where the vector fields are given by 

$$A_k (\theta )=(k_2 ,-k_1 )\cos k.\theta ,\quad B_k (\theta )=(k_2 ,-k_1 )\sin k.\theta$$
 and the Brownian motion

\begin{equation}
 \label{R3}
dW(t)=\sum_{k \in \ZZ }{\l_k\sqrt \nu}(A_k dx_k +B_k dy_k )
\end{equation}
 where $x_k ,y_k $ are independent copies of real Brownian motions. As in section~\ref{Section6} 
we assume that $\sum_k |k|^2 \l_k^2 <\infty$ and we consider
$\lambda_k =\lambda (|k|)$ to be nonzero for a equal number of $k_1$ and $k_2$
components. Again we write 
\begin{equation}
\label{R4} 
dg_t =(odW(t))+u(t, g_t )dt,\qquad d\tilde g_t =(odW(t))+u(t, \tilde g_t )dt
\end{equation}
with 

$$g_0 =\phi ,\qquad\tilde g_0 =\psi, \qquad \phi\not=\psi.$$

Changing the notation to $g_t=g_t(\theta)=x_t$, $\tilde g_t=\tilde g_t(\theta)=y_t$, we get
\begin{equation*}
\begin{split}
de(t)&=\f1{\rho_t(\theta)}\sum_{|k|\not=0}\l_k\sqrt{\nu}\left(\cos k\cdot \tilde g_t-\cos k\cdot g_t\right)k^{\bot, N}dx_k\\
&+\f1{\rho_t(\theta)}\sum_{|k|\not=0}\l_k\sqrt{\nu}\left(\sin k\cdot \tilde g_t-\sin k\cdot g_t\right)k^{\bot, N}dy_k\\
&+\f1{\rho_t(\theta)}\left((u^N(t,\tilde g_t)-u^N(t,g_t)\right)\,dt\\
&-\f1{2\rho_t^2(\theta)}\sum_{|k|\not=0}\l_k^2\nu |k^{\bot, N}|^2\left(\left(\cos k\cdot \tilde g_t-\cos k\cdot g_t\right)^2+\left(\sin k\cdot \tilde g_t-\sin k\cdot g_t\right)^2\right)e (t) \,dt\\
&=\f1{\rho_t(\theta)}\sum_{|k|\not=0}\l_k\sqrt{\nu}k^{\bot, N}\left(2\sin\f{k\cdot(\tilde g_t-g_t)}{2}\right)dz_k\\
&+\f1{\rho_t(\theta)}\left((u^N(t,\tilde g_t)-u^N(t,g_t)\right)\,dt\\
&-\f2{\rho_t^2(\theta)}\sum_{|k|\not=0}\l_k^2\nu |k^{\bot, N}|^2\sin^2 (\f{k\cdot(\tilde g_t-g_t)}{2}) e(t)\,dt
\end{split}
\end{equation*}
where $z_k$ is the Brownian motion defined by $$ dz_k=-\sin\f{k\cdot(\tilde g_t+g_t)}{2}\, dx_k+\cos\f{k\cdot(\tilde g_t+g_t)}{2}\, dy_k.$$
Noting $\di |k^{\bot, N}|^2=|k|^2 (n_k\cdot e(t))^2$, we obtain
\begin{equation}
 \label{R5}
\begin{split}
 de(t)&=\f1{\rho_t(\theta)}\sum_{|k|\not=0}|k|\l_k\sqrt{\nu}(n_k\cdot e(t))e'(t)\left(2\sin\f{k\cdot(\tilde g_t-g_t)}{2}\right)dz_k\\
&+\f1{\rho_t(\theta)}\left((u^N(t,\tilde g_t)-u^N(t,g_t)\right)\,dt\\
&-\f2{\rho_t^2(\theta)}\sum_{|k|\not=0}|k|^2\l_k^2\nu (n_k\cdot e(t))^2\sin^2\f{k\cdot(\tilde g_t-g_t)}{2} e(t)\,dt
\end{split}
\end{equation}
where $e'(t)$ is a unit vector in $\TT$ orthonormal to $e(t)$.
Now for every $K>0$,  if $\di \rho_t(\theta)\le \f{\pi}{2K}$ then for all $k$ such that $|k|\le K$,
$$
\f{\sin^2\f{k\cdot(\tilde g_t-g_t)}{2}}{|k|^2\rho_t^2(\theta)(n_k\cdot e(t))^2}\ge \f1{\pi^2}.
$$
Now using $|k|=|k^\bot|$ and $(n_k\cdot e(t))^2+(n_{k^\bot}\cdot e(t))^2=1$, we get 
\begin{equation}
 \label{R6}
\f2{\rho_t^2(\theta)}\sum_{|k|\not=0}|k|^2\l_k^2\nu (n_k\cdot e(t))^2\sin^2\f{k\cdot(\tilde g_t-g_t)}{2}\ge \f{\nu}{2\pi^2}\sum_{0<|k|<K}\l_k^2|k|^4.
\end{equation}
Observe that the term in the left is the second part of the drift in equation~\eqref{R5} as well as the derivative of the quadratic variation of $e(t)$. This yields the following result.
\begin{prop}
Identifying $T\TT$ with $\CC$, we have 
$\di 
e(t)=e^{iX_t}
$
where $X_t$ is a real-valued semimartingale with quadratic variation 
\begin{equation}\label{R7}
d[X,X]_t=\f4{\rho_t^2(\theta)}\sum_{|k|\not=0}|k|^2\l_k^2\nu (n_k\cdot e(t))^2\sin^2\f{k\cdot(\tilde g_t-g_t)}{2}\,dt
\end{equation}
and drift 
\begin{equation}\label{R8}
\int_0^t\f1{\rho_s(\theta)}\left\langle u(s,\tilde g_s)-u(s,g_s), ie(s)\right\rangle\,ds.
\end{equation}
We have for all $K>0$, on $\di \left\{\rho_t(\theta)\le \f{\pi}{2K}\right\}$,
\begin{equation}\label{R9}
d[X,X]_t\ge \f{\nu}{\pi^2}\sum_{0<|k|<K}\l_k^2|k|^4.
\end{equation}
 If 
$\di
\sum_{|k|\not=0}\l_k^2|k|^4=+\infty,
$
then as $\tilde g_t(\theta)$ gets closer and closer to $g_t(\theta)$, the rotation $e(t)$ becomes more and more irregular in the sense that the derivative of the quadratic variation of $X_t$ tends to infinity. 
\end{prop}

\medbreak\noindent
{\bf Acknowledgment.}
The second author wishes to thank the support of the Universit\'e de Poitiers. This work has also benefited from the portuguese grant\goodbreak PTDC/MAT/69635/2006.

%
%

\providecommand{\bysame}{\leavevmode\hbox to3em{\hrulefill}\thinspace}

\end{document}